\documentclass[reqno,12pt,centertags]{amsart}
\usepackage{geometry}
\usepackage{amsmath,amsthm,amscd,amssymb,latexsym,upref,stmaryrd}
\usepackage{graphicx}
\usepackage{hyperref}
\usepackage{enumitem}
\usepackage{xcolor}
\newcommand*{\mailto}[1]{\href{mailto:#1}{\nolinkurl{#1}}}
\usepackage[numbers,sort&compress]{natbib}

\newcommand{\beq}{\begin{equation}}
	\newcommand{\eeq}{\end{equation}}
\newcommand{\ba}{\begin{align}}
	\newcommand{\ea}{\end{align}}

\renewcommand{\Re}{\text{\rm Re}}
\renewcommand{\Im}{\text{\rm Im}}

 \allowdisplaybreaks
\numberwithin{equation}{section}

\newtheorem{theorem}{Theorem}[section]
\newtheorem{lemma}[theorem]{Lemma}
\newtheorem{corollary}[theorem]{Corollary}

\newtheorem{ip}[theorem]{Inverse Problem}
\theoremstyle{definition}

\newtheorem{proposition}[theorem]{Proposition}
\newtheorem{example}[theorem]{Example}

\begin{document}
	
	
	\title[Solving inverse problems by the method of spectral mappings ]
	{Solving Barcilon's inverse problems by the method of spectral mappings }
	
	\author[A.~W.~Guan]{AI-WEI GUAN}
	\address{Department of Mathematics, School of Mathematics and Statistics, Nanjing University of
		Science and Technology, Nanjing, 210094, Jiangsu, People's
		Republic of China}
	\email{\mailto{guan.ivy@njust.edu.cn}}
	
	\author[C.~F.~Yang]{CHUAN-FU Yang}
	\address{Department of Mathematics, School of Mathematics and Statistics, Nanjing University of
		Science and Technology, Nanjing, 210094, Jiangsu, People's
		Republic of China}
	\email{\mailto{chuanfuyang@njust.edu.cn}}
	
	\author[N.~P.~Bondarenko]{NATALIA P. BONDARENKO}
	\address{S.M. Nikolskii Mathematical Institute, Peoples' Friendship University of Russia (RUDN University), 6 Miklukho-Maklaya Street, Moscow, 117198, Russian Federation}
	\email{\mailto{bondarenkonp@sgu.ru}}

	\date{\today}
	
	\begin{abstract}
		{
			In this paper, we consider Barcilon's inverse problem, which consists in the recovery of the fourth-order differential operator from three spectra. We obtain the relationship of Barcilon's three spectra with the Weyl-Yurko matrix. Moreover, we prove the uniqueness theorem for the inverse problem solution by developing the ideas of the method of spectral mappings. Our approach allows us to obtain the result for the general case of complex-valued distributional coefficients. In the future, the methods and the results of this paper can be generalized to differential operators of orders greater than $4$ and used for further development of the inverse problem theory for higher-order differential operators.
		}
	\end{abstract}
	
	\maketitle
	
	\section{Introduction}
	This paper deals with an inverse spectral problem for the fourth-order operators generated by the following differential expression:
	\begin{equation}\label{ident2}
		l(y)=y^{(4)}-(py')' +qy, \quad x \in (0,1),
	\end{equation}
	with distribution coefficients $p\in W_{2}^{-1}(0,1)$ and $q\in W_{2}^{-2}(0,1)$. The derivatives in (\ref{ident2}) are understood in the sense of distributions.
	
	Inverse problems of spectral analysis consist in the recovery of operators from their spectral characteristics. The most complete results in the inverse problem theory were obtained for the second-order Sturm-Liouville operators $-y'' + q(x) y$ (see, e.g., the monographs by Levitan \cite{Lev87}, Marchenko \cite{Mar86}, Freiling and Yurko \cite{FY01}, Kravchenko \cite{Krav20}, and references therein). In particular, Borg \cite{Bor} in 1946 has proved that the potential $q(x)$ of the Sturm-Liouville equation
	$$
		-y'' + q(x) y = \lambda y, \quad x \in (0, 1),
	$$
	is uniquely specified by two spectra corresponding to different sets of boundary conditions, say, $y(0) = y(1) = 0$ and $y'(0) = y(1) = 0$. A constructive method for solving the inverse Sturm-Liouville problem was suggested by Gelfand and Levitan \cite{GL51}. However, the Gelfand-Levitan method appeared to be ineffective for differential operators of orders greater than $2$. Thus, the higher-order differential operators are fundamentally more difficult for investigation than the second-order ones, so the inverse spectral theory for higher orders still contains many open problems.
	
	In 1974, Barcilon \cite{Bar1, Bar2} studied an inverse problem for the fourth-order differential equation
	\begin{equation} \label{eqv-Bar}
		y^{(4)} - (p y')' + q y = \lambda y, \quad x \in [0,1],
	\end{equation}
	where the functions $p(x)$ and $q(x)$ are real-valued, $p(x)$ is differentiable, and $\lambda$ is the spectral parameter. Barcilon's problem consists in the reconstruction of the functions $p$ and $q$ from three spectra and the value $p(0)$. As an example, one can use the spectra of the boundary value problems for equation \eqref{eqv-Bar}
	subject to the following boundary conditions:
\begin{equation}\label{four}
	\begin{split}
	& y(0) = y'(0) = 0, \quad y(1) = y''(1) = 0, \\
	& y(0) = y''(0) = 0, \quad y(1) = y''(1) = 0, \\
	& y'(0) = y''(0) = 0, \quad y(1) = y''(1) = 0.		
    \end{split}	
\end{equation}

	Thus, the boundary conditions at $x = 1$ for the three spectra coincide and the boundary condition at $x = 0$ differ. In the first paper \cite{Bar1}, the uniqueness of recovering $p$ and $q$ was proved. In the second paper \cite{Bar2}, Barcilon presented a reconstruction procedure based on iterations. However, the research of Barcilon \cite{Bar1, Bar2} was mostly focused on applications in geophysics, so some aspects in his papers lack mathematical rigor. In particular, Barcilon did not specify to which classes do the functions $p(x)$ and $q(x)$ belong. As far as the authors understand, the uniqueness result of \cite{Bar1} is valid for $p \in W_1^1[0,1]$ and $q \in L_1(0,1)$, and the reconstruction procedure \cite{Bar2} requires a higher order of smoothness. Moreover, Barcilon did not investigate existence and stability of the inverse problem solution, so these issues remain open. 
	
	Spectral theory of the fourth-order differential operators causes interest of scholars because of applications in mechanics, optics, acoustics, and other fields of science (see \cite{Mik, Pol}). In recent years, various spectral theory issues for such operators with non-smooth coefficients were considered in \cite{poly,bada,ugur}. In the recent study \cite{pere}, Perera and B\"ockmann developed a numerical technique for solving Barcilon's problem. Inverse problems  for the fourth-order differential operators in other statements were investigated in \cite{Mc1, Mc2,PK,Cau,Glad,mora,xx}.

	The general theory of inverse spectral problems for the higher-order differential operators
	$$
		y^{(n)} + \sum_{k = 0}^{n-2} p_k(x) y, \quad n > 2,
	$$
	has been constructed by Yurko \cite{Yur,Yur1,Yur2}. In the authors' opinion, the main achievement of Yurko in these studies was the right choice of the spectral characteristic that uniquely determined the operator in the general case, with no restrictions on its spectrum. This characteristic is the Weyl-Yurko matrix $M(\lambda)$, which is a meromorphic triangular matrix-function generalizing the $m$-function of the Sturm-Liouville operators (see, e.g., \cite{Mar86, Yur}). One of this paper's goals is to show that Barcilon's problem, in fact, is a special case of the inverse problem by the Weyl-Yurko matrix. Another important achievement of Yurko \cite{Yur} is the development of the method of spectral mappings. This method can be applied not only to the second-order operators but also to the higher-order ones, so it is more universal than the Gelfand-Levitan method.
				
	In the last 20 years, inverse problem theory has been actively developed for differential operators with distributional coefficients. For the Sturm-Liouville operators, Hryniv and Mykutyuk \cite{Hry1,Hry2,Hry3} transferred the transformation operator method and generalized the basic results of inverse problem theory to distributional potentials of class $W_2^{-1}(0,1)$. Note that the space $W_2^{-1}(0,1)$ contains potentials with singularities such as the Dirac $\delta$-functions, the Coulumb-type singularities $\frac{1}{x}$, etc.
	The method of spectral mappings has been extended to Sturm-Liouville operators with  potentials of $W_2^{-1}(0,1)$ by Freiling et al \cite{Fre2} and by Bondarenko \cite{Bon4,Bon5}. In the last few years, the investigation of higher-order differential operators with distribution coefficients has gradually advanced. Mirzoev and Shkalikov \cite{Mir1,Mir2} have developed a regularization approach to such operators. For the differential equation of arbitrary even order with singular coefficients, Savchuk and Shkalikov \cite{Sav} have constructed the Birkhoff solutions with a certain asymptotic behavior as the spectral parameter tends to infinity. The investigation of inverse spectral problems for higher-order differential operators with distributional coefficients was started by Bondarenko \cite{Bon2, Bon3, Bon1, Bon7}. The papers \cite{Bon2, Bon3} are concerned with the uniqueness theorems for such operators on a finite interval and on the half-line. In \cite{Bon1}, Bondarenko has developed a constructive approach to the recovery of differential operator coefficients from the spectral data by relying on the ideas of the method of spectral mappings. This approach can be applied to various classes of higher-order differential operators with integrable or distributional coefficients. In \cite{Bon7}, by using this approach, the necessary and sufficient conditions on the spectral data of the third-order differential operator with distribution coefficient were obtained.
	
	In this paper, following the ideas of Barcilon \cite{Bar1, Bar2}, we consider the problem of recovering the fourth-order operator \eqref{ident2} from three spectra and investigate the relationship of this problem with the method of spectral mappings \cite{Yur, Bon1}. In contrast to Barcilon's papers \cite{Bar1, Bar2}, we consider the differential expression \eqref{ident2} in a more general situation.
	We suppose that the functions $p(x)$ and $q(x)$ are complex-valued and belong to distribution spaces $W_2^{-1}(0,1)$ and  $W_2^{-2}(0,1)$, respectively. We provide the problem statement for these functional classes in terms of the regularization approach of Mirzoev and Shkalikov \cite{Mir1}. 
	Then, we investigate the relationship of the three Barcilon-type spectra with the Weyl-Yurko matrix. Finally, we prove the uniqueness theorem for the inverse problem solution by using the method of spectral mappings. In the future, our results and methods can be generalized to differential operators of orders greater than $4$ and also can be used for investigation of existence and stability issues.
	
	The paper is organized as follows. In Section~\ref{sec:main}, we describe the regularization of the differential expression \eqref{ident2}, introduce the basic concepts, and provide the main results. Section~\ref{sec:prelim} contains some preliminaries and auxiliary lemmas. In Section~\ref{sec:proofs}, we prove the main theorems. In Section~\ref{sec:different}, we discuss the application of our results to the cases of lower singularity orders.
	
	\section{Main results}\label{sec:main} 
	
	Let us start with the regularization of the differential expression \eqref{ident2}.
	Our treatment of this expression is based on the approach of Mirzoev and Shkalikov \cite{Mir1}. This approach consists in the construction of the matrix function $F(x)=[f_{k,j}(x)]_{k,j=1}^{4}$ which is associated with the differential expression (\ref{ident2}) by a specific rule. This rule depends on the order of the differential equation and on the classes of coefficients.
	
	For the fourth-order differential expression \eqref{ident2}, the associated matrix for the first time appeared in the short note of Vladimirov \cite{Vlad04} and has the form
	\begin{equation}\label{ident3}
		F(x)=\begin{pmatrix} 0 & 1 & 0 & 0\\ -\tau_{2} & \tau_{1} & 1 & 0 \\ \tau_{1}\tau_{2} & -\tau_{1}^{2}+2\tau_{2} & -\tau_{1} & 1 \\ \tau_{2}^{2} & -\tau_{1}\tau_{2} & -\tau_{2} & 0\end{pmatrix},
	\end{equation}
	where $\tau_{1}'=p$, $\tau_{2}''=q$, $\tau_1, \tau_2 \in L_2(0,1)$.
	
	Using the elements of the matrix function $F(x)$ defined by (\ref{ident3}), we introduce the quasi-derivatives
	\begin{equation}\label{ident5}
		y^{[0]}:=y,\quad y^{[k]}=(y^{[k-1]})'-\sum\limits_{j=1}^{k}f_{k,j}y^{[j-1]}, \quad k = \overline{1,4},
	\end{equation}
	and the domain
	\begin{align*}
		\mathcal{D}_{F}:=\{y \colon y^{[k]}\in AC[0,1],\;k=\overline{0,3}\}.
	\end{align*}
	
	The matrix $F(x)$ is constructed in such a way that, for any $y \in \mathcal{D}_{F}$, the expression $l(y)$ produces a regular generalized function and
	the relation $l(y)=y^{[4]}$ holds. In particular, we call the function $y$ a solution of the equation
	\begin{equation}\label{ident4}
		l(y)=\lambda y, \quad x \in (0,1),
	\end{equation}
	where $\lambda$ is the spectral parameter, if $y \in \mathcal{D}_{F}$ and $y^{[4]}=\lambda y$, $x \in (0,1)$.
	
	Proceed to the inverse problem statement. By using the matrix $F(x)$ defined by (\ref{ident3}) and the corresponding quasi-derivatives (\ref{ident5}), we define the linear forms
	\begin{equation}\label{ident6}
		\begin{split}
			&U_{s}(y)=y^{[s-1]}(0), \quad s = \overline{1,4},\\
			&V_{1}(y)=y^{[3]}(1), \quad V_{2}(y)=y^{[1]}(1), \quad
			V_{3}(y)=y^{[2]}(1), \quad V_{4}(y)=y(1).
		\end{split}
	\end{equation}
	
	For $(i,j)\in \{(1,2), (1,3), (2,3)\}$, denote by $\mathfrak{S}_{ij}$ the spectrum of the boundary value problem for equation (\ref{ident4}) subject to the boundary conditions
	\begin{equation} \label{bc-Bar}
	U_{i}(y)=U_{j}(y)=0, \quad V_{3}(y)=V_{4}(y)=0.
	\end{equation}
	
	It can be shown by the standard method that the spectra $\mathfrak S_{ij}$ are countable sets of eigenvalues (see \cite{Bon2}).
	Thus, we consider the boundary conditions similar to the ones that were used by Barcilon \cite{Bar2} and that are provided in the Introduction of this paper. However, the linear forms \eqref{ident6} are chosen for definiteness. They can be replaced by linear forms of other orders. The inverse problem is formulated as follows.
	
	\begin{ip} \label{ip:main}
		Given the spectra $\mathfrak S_{12}$, $\mathfrak S_{13}$, and $\mathfrak S_{23}$, find $p$ and $q$.
	\end{ip}

	Furthermore, we introduce the following coefficient matrices related to the linear forms \eqref{ident6}:
	\begin{equation*}\label{U,V}
		U=I, \qquad
		V=\begin{pmatrix} 0 & 0 & 0 & 1\\ 0 & 1 & 0 & 0 \\ 0 & 0 & 1 & 0 \\ 1 & 0 & 0 & 0\end{pmatrix},
	\end{equation*}
	where $I$ is the $(4 \times 4)$-unit matrix.
	
	Below, we call by the problem $\mathcal L$ the triple $(F(x), U, V)$ and introduce some notations related to this problem.
	
	Denote by $\{C_{k}(x,\lambda)\}_{k=1}^{4}$ , $\{S_{k}(x,\lambda)\}_{k=1}^{4}$, and $\{\Phi_{k}(x,\lambda)\}_{k=1}^{4}$ the solutions of the equation (\ref{ident4}) satisfying the initial conditions
	\begin{equation*}
		U_{s}(C_{k})=\delta_{s,k}, \quad s = \overline{1,4}, 
	\end{equation*}
	\begin{equation}\label{BV}
		V_{s}(S_{k})=\delta_{s,k}, \quad s = \overline{1,4},
	\end{equation}
	and the boundary conditions
	\begin{equation*}
		U_{s}(\Phi_{k})=\delta_{s,k}, \quad s= \overline{1,k}, \quad
		V_{s}(\Phi_{k})=0, \quad s=\overline{k+1,4},
	\end{equation*}
	respectively, where $\delta_{s,k}$ is the Kronecker delta. 
	The solutions $\{C_{k}(x,\lambda)\}_{k=1}^{4}$ , $\{S_{k}(x,\lambda)\}_{k=1}^{4}$ and their quasi-derivatives are entire in $\lambda$ for each fixed $x \in [0,1]$, and $\{\Phi_{k}(x,\lambda)\}_{k=1}^{4}$ are called the Weyl solutions of the  problem $\mathcal{L}$.
	
	Introduce the notation $\vec{y}(x)=\text{col}(y^{[0]}(x), y^{[1]}(x), y^{[2]}(x), y^{[3]}(x))$ and the $(4\times 4)$-matrices $C(x,\lambda)=[\vec C_{k}(x,\lambda)]_{k=1}^{4}$, $\Phi(x,\lambda)=[\vec \Phi_{k}(x,\lambda)]_{k=1}^{4}$. It has been shown in \cite{Bon2} that the following relation holds:
	\begin{equation}\label{M'}
		\Phi(x,\lambda)=C(x,\lambda)M(\lambda),
	\end{equation}
	where the matrix function $M(\lambda)$ is called the Weyl-Yurko matrix of the problem $\mathcal{L}$. 
	
	Due to the results of \cite{Bon2}, the Weyl-Yurko matrix $M(\lambda)=[m_{jk}(\lambda)]_{j,k=1}^{4}$ is unit lower-triangular and its non-trivial entries have the form
	\begin{equation}\label{M}
		m_{jk}(\lambda)=-\frac{\Delta_{jk}(\lambda)}{\Delta_{kk}(\lambda)}, \quad 1\leq k \;\textless j \leq 4,
	\end{equation}
	where $\Delta_{kk}(\lambda):=\text{det}[V_{s}(C_{r})]_{s,r=k+1}^{4}$, and $\Delta_{jk}(\lambda)$ is obtained from $\Delta_{kk}(\lambda)$ by the replacement of $C_{j}$ by $C_{k}$. One can easily show that the zeros of the functions $\Delta_{jk}(\lambda)$ coincide with the eigenvalues of the boundary value problems $\mathcal L_{jk}$ for equation (\ref{ident4}) with the following boundary conditions: 
	\begin{equation} \label{bcL}
	U_{\xi}(y)=0, \quad \xi=\overline{1,k-1},j, \quad V_{\eta}(y)=0, \quad \eta=\overline{k+1,4}.
	\end{equation}  
	
	The functions $C_{r}^{[s]}(1,\lambda)$, $r=\overline{1,4}$, $s=\overline{0,3}$, are entire analytic in $\lambda$, so do the functions $\Delta_{jk}(\lambda)$, $1\leq k \leq j \leq 4$. Hence $M(\lambda)$ is meromorphic in $\lambda$, and the poles of the $k$-th column of $M(\lambda)$ coincide with zeros of $\Delta_{kk}(\lambda)$. 
	
	We will say that the problem $\mathcal{L}$ belongs to the class $\mathcal{W}$ if all the zeros of $\Delta_{kk}(\lambda)$ are simple for $k=1,2,3$ and that $\Delta_{kk}(\lambda)$ and $\Delta_{k+1,k+1}(\lambda)$ do not have common zeros for $k=1,2$. Then, in view of (\ref{M}), the poles of $M(\lambda)$ are simple. In this paper, we always assume that $\mathcal{L} \in \mathcal{W}$.
			
	Consider another problem $\tilde{\mathcal{L}}=(\tilde{F}(x), {U}, {V})$, where $\tilde{F}(x)$  is a matrix function of the same form as ${F}(x)$ but with different coefficients $\tilde{\tau}_{1},\tilde{\tau}_{2}\in L_{2}(0,1)$. We agree that, if $\xi$ represents an object related to $\mathcal{L}$, then $\tilde{\xi}$ represents the analogous object related to $\tilde{\mathcal{L}}$. We assume that $\mathcal{L}=(F(x),U , V)\in \mathcal{W}$ and $\tilde{\mathcal{L}}=(\tilde{F}(x), U ,V)\in \mathcal{W}$. The main theorem of this paper is formulated as follows:
	
	\begin{theorem}\label{thm:main}
		If $\mathfrak{S}_{12}=\tilde{\mathfrak{S}}_{12}$, $\mathfrak{S}_{13}=\tilde{\mathfrak{S}}_{13}$, $\mathfrak{S}_{23}=\tilde{\mathfrak{S}}_{23}$, then $\tau_{1}=\tilde{\tau}_{1},\;\tau_{2}=\tilde{\tau}_{2}+c x$, where $c$ is a complex constant, so
		$p=\tilde{p}$ in $W_2^{-1}(0,1)$ and $q=\tilde{q}$ in $W_2^{-2}(0,1)$ Thus, the specification of the three spectra $\mathfrak{S}_{12},\mathfrak{S}_{13},\mathfrak{S}_{23}$ uniquely determines the coefficients $p$ and $q$ of the differential expression $l(y)$.
	\end{theorem}

	
	In Theorem \ref{thm:main}, we only have $p=\tilde{p}$, $q=\tilde{q}$. We lose the uniqueness of $\tau_2\in L_2(0,1)$. A natural question is  whether we can determine $\tau_2$ by the three spectra $\mathfrak{S}_{12},\mathfrak{S}_{13},\mathfrak{S}_{23}$. The following theorem gives the negative answer to this question. 	
	
	\begin{theorem}\label{thm:main222}
		If    $\tau_{1}=\tilde{\tau}_{1}$ and $\tau_{2}=\tilde{\tau}_{2}+c x$, where $c \in \mathbb C$ is an arbitrary constant, then $\mathfrak{S}_{12}=\tilde{\mathfrak{S}}_{12}$, $\mathfrak{S}_{13}=\tilde{\mathfrak{S}}_{13}$, $\mathfrak{S}_{23}=\tilde{\mathfrak{S}}_{23}$.
	\end{theorem}
	
	\section{Preliminaries} \label{sec:prelim}
	In this section, we provide some preliminaries. We start with the asymptotic properties of the characteristic functions.
	First, consider the sectors
	\begin{equation*}
		\Gamma_{k} :=\left\{\rho\colon \frac{\pi(k-1)}{4}\;\textless\; \arg\rho\;\textless \;\frac{\pi k}{4} \right\}, \quad k=\overline{1,8},
	\end{equation*}
    and the region
    \begin{equation*}
    	\mathcal{G}_{\delta}=\mathcal  G_{\delta, ijk} := \{\rho\in\overline{\Gamma}_{k}:|\rho-\rho_{ij,l}|\geq\delta,\;l \geq 1\}, \quad \delta\textgreater 0,
    \end{equation*}
    where $\{\rho_{ij,l}\}_{l\geq 1}$ are the zeros of $\Delta_{ij}(\rho^4)$ in the $\rho$-plane.
   
	If $\rho$ lies in a fixed sector $\Gamma=\Gamma_{k}$, we denote by $\{\omega_{j}\}_{j=1}^{4}$ the roots of equation $\omega^{4}=1$ which are numbered in the following order 
	\begin{equation*}
		\Re(\rho \omega_{1})\;\textless\;\Re(\rho \omega_{2})\;\textless\;\Re(\rho \omega_{3})\;\textless\;\Re(\rho \omega_{4}), \quad \rho\;\in\;\Gamma.
	\end{equation*}
	
	\begin{proposition}\label{delta}
		(\cite{Bon2}) The following asymptotic relation holds as 
		$|\rho|\rightarrow\infty:$
		\begin{equation*}
			\Delta_{ij}(\lambda)=c_{ij}\rho^{a_{ij}}e^{\rho s_{j}}(1+o(1)), \quad \arg\rho=\phi, \quad 1\leq j < i\leq 4,
		\end{equation*}
		where $\lambda=\rho^4$, $\{\rho:\arg\rho=\phi\}\subset \Gamma$ is a fixed ray, $s_{j}=\sum\limits_{k=j+1}^{4}\omega_{k}$,
		$$a_{ij}=\sum_{k=1}^{j-1}p_{k,0}+\sum_{k=j+1}^{4}p_{k,1}+p_{i,0}-6,$$ $$c_{ij}=(\det\Omega)^{-1}\det[\omega_{k}^{p_{s,0}}]_{k=\overline{1,j},s=\overline{1,j-1},i}\cdot \det[\omega_{k}^{p_{s,1}}]_{s,k=j+1}^{4},$$ $\Omega:=[\omega_{k}^{j-1}]_{j,k=1}^{4}$, $p_{k,0}$ and $p_{k,1}$ are the orders of the linear forms $U_k$ and $V_k$, respectively, for $k = \overline{1,4}$, that is, $p_{1,0}=p_{4,1}=0,\;p_{2,0}=p_{2,1}=1,\;p_{3,0}=p_{3,1}=2,\;p_{4,0}=p_{1,1}=3$.
	\end{proposition}
    Moreover, considering the sector with boundary, we have the following asymptotic estimates:
    \begin{proposition}\label{delta'}
    	(\cite{Bon6}) For sufficiently large $|\rho|$, the following relations hold:
    	\begin{equation*}
    		\begin{array}{ll}
    			\Delta_{ij}(\lambda)=O(\rho^{a_{ij}}e^{\rho s_{j}}),\quad \rho\in\overline{\Gamma},\\
    			|\Delta_{ij}(\lambda)|\geq C_{\delta}|\rho|^{a_{ij}}e^{Re(\rho s_{j})},\quad\rho\in\mathcal{G}_{\delta},
    		\end{array}
    	\end{equation*}
    	where $\lambda=\rho^4$, $C_{\delta}$ is a constant depending on $\mathcal{G}_{\delta}$, $a_{ij},\;s_{j},\;1\leq j < i\leq 4$ are defined in Proposition \ref{delta}.
    \end{proposition}

	Next, following \cite{Bon1}, we define the auxiliary problem $\mathcal{L}^{\star}=(F^{\star}(x),U^{\star},V^{\star})$, which helps us to investigate the structure of the Weyl-Yurko matrix.
	Along with the associated matrix $F(x)$, we consider the matrix function $F^\star(x)=[f^\star_{k,j}(x)]^{4}_{k,j=1}$ whose entries are defined as follows:
	\begin{equation*}\label{ident f^*_{k,j}(x)}
		f^\star_{k,j}(x):=(-1)^{k+j+1}f_{5-j,5-k}(x).
	\end{equation*}
	
	Using $F^{\star}(x)$, we define the quasi-derivatives 
	\begin{equation}\label{identz}
z^{[0]}:=z, \quad z^{[k]}=(z^{[k-1]})'-\sum_{j=1}^{k}f_{k,j}^{\star}z^{[j-1]}, \quad k = \overline{1,4},
\end{equation}
	and the domain
	\begin{equation*}
\mathcal{D}_{F^{\star}}:=\{z:z^{[k]}\in AC[0,1], \quad k= \overline{0,3}\}.
\end{equation*}
	
	Suppose that  $y \in \mathcal{D}_{F}$ and $z \in \mathcal{D}_{F^\star}$, the quasi-derivatives for $y$ are defined via (\ref{ident5}) by using the matrix $F(x)$, the quasi-derivatives for $z$ are defined by \eqref{identz}.
	Define the Lagrange bracket
	\begin{equation} \label{Lagrange}
		\langle z,y\rangle:=zy^{[3]}-z^{[1]}y^{[2]}+z^{[2]}y^{[1]}-z^{[3]}y.
	\end{equation}
	
	It can be easily shown that, if $y$ and $z$ solve the equations $y^{[4]} = \lambda y$ and $z^{[4]} = \mu z$, respectively, then
	\begin{equation} \label{wron}
		\frac{d}{dx} \langle z,y\rangle = (\lambda - \mu) z y.
	\end{equation}
	In particular, if $\lambda = \mu$, then the Lagrange bracket $\langle z,y\rangle$ does not depend on $x$.
	
	Set 
	$$
	\vec{y}(x)=\text{col}(y^{[0]}(x), y^{[1]}(x), y^{[2]}(x), y^{[3]}(x)), \quad\!\!\!\! \vec{z}(x)=\text{col}(z^{[0]}(x), z^{[1]}(x), z^{[2]}(x), z^{[3]}(x)).
	$$

	By using the corresponding quasi-derivatives (\ref{ident5}) and (\ref{identz}). Then, we have the following relation:
	\begin{equation*}\label{ident J}
		\langle z,y\rangle |_{x=a}=[\vec{z}(a)]^{T}J\vec{y}(a),
	\end{equation*}
	where $J:=[(-1)^{k+1}\delta_{k,5-j}]^{4}_{k,j=1}$. 
	
	Along with $U$ and $V$, consider the matrices
	\begin{equation*}\label{ident UV}
		U^{\star}:=[JU^{-1}J_{0}^{-1}]^{T}, \quad V^{\star}:=[JV^{-1}J_{1}^{-1}]^{T},
	\end{equation*}
	where 
	\begin{equation}\label{ident J0J1}
		J_{0}=\begin{pmatrix} 0 & 0 & 0 & 1\\ 0 & 0 & -1 & 0 \\ 0 & 1 & 0 & 0 \\ -1 & 0 & 0 & 0\end{pmatrix}, \quad J_{1}=\begin{pmatrix} 0 & 0 & 0 & -1\\ 0 & 0 & -1 & 0 \\ 0 & 1 & 0 & 0 \\ 1 & 0 & 0 & 0\end{pmatrix},
	\end{equation}
	and the sign $T$ denotes the matrix transpose.
	
	Analogously to the Weyl-Yurko matrix $M(\lambda)$, one can define the matrix $M^{\star}(\lambda)$ for the problem $\mathcal L^{\star} = (F^{\star}(x), U^{\star}, V^{\star})$. In \cite{Bon1}, a relation between the Weyl-Yurko matrices $M(\lambda)$ and $M^{\star}(\lambda)$ has been obtained. However, in our case, calculations show that
	$F(x) = F^\star(x)$, $U = U^\star$, and $V = V^\star$. Hence $M(\lambda)=M^\star(\lambda)$. Consequently, the results of \cite{Bon1} imply the following proposition.
	
	\begin{proposition}
		(\cite{Bon1}). The following relation holds:
		\begin{equation}\label{M*}
			(M(\lambda))^{T}J_{0}M(\lambda)=J_{0}.
		\end{equation}
	\end{proposition}

	In the element-wise form, we obtain the following corollary.
	
	\begin{corollary}\label{real}
		The following relations hold:
		\begin{equation}\label{real1}
			m_{43}(\lambda)={m}_{21}(\lambda),
		\end{equation}
		\begin{equation}\label{real2}
			m_{42}{(\lambda)}-m_{32}(\lambda)m_{21}(\lambda)+m_{31}(\lambda)=0.
		\end{equation}
	\end{corollary}
	\begin{proof}
		We know that $M$ is a unit lower-triangular matrix. Substituting (\ref{ident J0J1}) into (\ref{M*}), we obtain
		\begin{equation*}
			\begin{pmatrix} 0 & m_{31}+m_{42}-m_{21}m_{32} & m_{43}-m_{21} & 1\\ m_{21}m_{32}-m_{31}-m_{42} & 0 & -1 & 0 \\ m_{21}-m_{43} & 1 & 0 & 0 \\ -1 & 0 & 0 & 0\end{pmatrix}\!\!\!=\!\!\!\begin{pmatrix} 0 & 0 & 0 & 1\\ 0 & 0 & -1 & 0 \\ 0 & 1 & 0 & 0 \\ -1 & 0 & 0 & 0\end{pmatrix}.
		\end{equation*}
		
		This yields the relations (\ref{real1}) and (\ref{real2}).\end{proof}
	
	Let us show that the spectra $\mathfrak S_{ij}$ coincide with the zeros of certain characteristic functions.
	
	\begin{lemma}
		The spectra $\mathfrak{S}_{12}$, $\mathfrak{S}_{13}$, $\mathfrak{S}_{23}$  coincide with the zeros of the characteristic functions $\Delta_{22}$, $\Delta_{32}$, $\Delta_{42}$, respectively.
	\end{lemma}
	
	\begin{proof}
		Recall that the zeros of $\Delta_{jk}(\lambda)$ coincide with the eigenvalues of the boundary value problem $\mathcal L_{jk}$ for equation \eqref{ident4} subject to the boundary conditions \eqref{bcL}. Clearly, the boundary conditions \eqref{bcL} for $j = k = 2$ and for $j = 3$, $k = 2$ coincide with \eqref{bc-Bar} for $i = 1$, $j = 2$ and for $i = 1$, $j = 3$, respectively. Hence, the zeros of $\Delta_{22}$ and $\Delta_{32}$  coincide with the spectra  $\mathfrak{S}_{12}$ and $\mathfrak{S}_{13}$, respectively. We only need to prove zeros of the function $\Delta_{42}$ coincide with the spectrum $\mathfrak{S}_{23}$.
		
		Recall that $S_1(x, \lambda)$ and $S_2(x, \lambda)$ satisfy the equation $y^{[4]} = \lambda y$ and $\mathcal L = \mathcal L^{\star}$. Therefore, using \eqref{BV} and \eqref{wron}, we obtain
		\begin{equation} \label{ident}
			\langle S_{1},S_{2}\rangle|_{x=0}=\langle S_{1}, S_{2}\rangle|_{x=1}=0.
		\end{equation}
		Using \eqref{ident6}, \eqref{Lagrange}, and \eqref{ident}, we get
		\begin{equation}\label{ident1}
			U_{2}(S_{1})U_{3}(S_{2})-U_{3}(S_{1})U_{2}(S_{2})=U_{1}(S_{1})U_{4}(S_{2})-U_{4}(S_{1})U_{1}(S_{2}).
		\end{equation}
		
		Obviously, the left-hand side of \eqref{ident1} is the characteristic function of the boundary value problem for equation \eqref{ident4} subject to the boundary conditions
		$$
			U_2(y) = U_3(y) = 0, \quad V_3(y) = V_4(y) = 0.
		$$
		Thus, its zeros coincide with the spectrum $\mathfrak{S}_{23}$. The right-hand side of \eqref{ident1} is the characteristic function of the boundary value problem $\mathcal L_{42}$. Hence, the spectrum $\mathfrak S_{32}$ coincide with the zeros of $\Delta_{42}(\lambda)$.
		\end{proof}
	
	\begin{lemma}\label{char}
		The  spectra  $\mathfrak{S}_{12},\mathfrak{S}_{13},\mathfrak{S}_{23}$ uniquely determine the characteristic functions $\Delta_{22}$, $\Delta_{32}$, $\Delta_{42}$, respectively.
	\end{lemma}
	\begin{proof}
		Let $\Lambda_{m,k}=\{\lambda_{lmk}\}_{l\geq1}$ be the zeros of $\Delta_{mk}(\lambda)$, where $(m,k)$ is the index in the set $\{(2,2),(3,2),(4,2)\}$. It follows from {Proposition \ref{delta'}} that the function $\Delta_{mk}(\lambda)$ is entire in $\lambda$ of the order $\frac{1}{4}$. By Hadamard's factorization theorem, we have
		\begin{equation*}
			{\Delta}_{mk}(\lambda)=C_{mk}\prod_{l=1}^{\infty}\left(1-\frac{\lambda}{\lambda_{lmk}}\right),
		\end{equation*}
		where $C_{mk}$ are some constants.
		
		By virtue of Proposition \ref{delta}, we have the following asymptotic relation:
		\begin{equation*}
			\begin{split}
				\Delta_{mk}(\lambda)=c_{mk}\rho^{a_{mk}}e^{\rho (1-i)}(1+o(1)), \quad \arg\rho=\frac{\pi}{8}, \quad |\rho|\rightarrow\infty,
			\end{split}
		\end{equation*}
		where $c_{22}=\frac{1-i}{8}, c_{32}=\frac{i}{4}, c_{42}=-\frac{1+i}{8},a_{22}=-3,=a_{32}=-2,a_{42}=-1.$
		Hence, we obtain
		\begin{equation*}
			\begin{split}
				C_{mk}=\lim_{\arg\lambda=\frac{\pi}{2} \atop |\lambda|\rightarrow{\infty}} c_{mk}\rho^{a_{mk}}e^{\rho (1-i)}\prod_{l=1}^{\infty}\left(1-\frac{\lambda}{{\lambda}_{lmk}}\right)^{-1}. 
			\end{split}
		\end{equation*}
		This yields the claim.\end{proof}
	
	For convenience, we introduce the following notation.
	If for $\lambda\rightarrow\lambda_{0},$
		$$A(\lambda)=\sum_{k=-q}^{p}a_k(\lambda-\lambda_{0})^k+o((\lambda-\lambda_{0})^p),$$
		then $A_{\langle k\rangle}(\lambda_0):=a_k.$
	
		Along with the Weyl-Yurko matrix, we also consider the discrete spectral data of the problem $\mathcal L$.
	Denote by $\Lambda_k=\{\lambda_{nk}\}_{n\geq 1}$ the set of zeros of  $\Delta_{kk}$ for $k=1,2,3$, $\Lambda=\cup_{k=1}^3\Lambda_k$ the set of the poles of the Weyl matrix $M(\lambda)$. Consider the Laurent series
	\begin{equation*}
	M(\lambda)=\frac{M_{\langle-1\rangle}(\lambda_{0})}{\lambda-\lambda_{0}}+M_{\langle0\rangle}(\lambda_{0})+M_{\langle1\rangle}(\lambda_{0})(\lambda-\lambda_{0})+\dots, \quad \lambda_{0} \in \Lambda.
	\end{equation*}
	Define the weight matrices
	\begin{equation}\label{N}
	\mathcal{N}(\lambda_{0}):=(M_{\langle0\rangle}(\lambda_{0}))^{-1}M_{\langle-1\rangle}(\lambda_{0}), \quad \lambda_{0} \in \Lambda.
	\end{equation}
	
	Consider the entires of the matrix $\mathcal{N}(\lambda_0)=[\mathcal{N}_{kj}(\lambda_0)]_{k,j=1}^4.$ Since $M(\lambda)$ is unit lower-triangular, we have $\mathcal{N}_{kj}(\lambda_0)=0$ for all $k\leq j,\;\lambda_{0}\in\Lambda.$ Moreover, in  view of Lemma~2.4 of \cite{Bon1}, we also have:
	
    \begin{proposition} \label{prop:N}
		Suppose that $\mathcal{L} \in \mathcal{W},$ we have the following properties:
		
		(I) If $\lambda_{0}\notin \Lambda_k$, then $\mathcal{N}_{sj}(\lambda_0)=0,\;s=\overline{k+1,4},\;j=\overline{1,k}.$
		
		(II) If $\lambda_{0}\in \Lambda_s$ for $s=\overline{\nu+1,k-1},\;\lambda_{0}\notin \Lambda_\nu,\;\lambda_{0}\notin \Lambda_k,\;1\leq\nu+1\;\textless\;k \leq 4 $, then $\mathcal{N}_{k,\nu+1}(\lambda_0)\neq 0.$ (Here $\Lambda_0=\Lambda_4=\varnothing$).
	\end{proposition}
	
	In the proof of the main theorem, an important role is played by the matrix of spectral mappings $\mathcal{P}(x, \lambda)$, which is defined as follows:
	\begin{equation} \label{defP}
	\mathcal P(x, \lambda) := \Phi(x, \lambda) (\tilde \Phi(x, \lambda))^{-1}.	
	\end{equation}
	Thus, $\mathcal P(x, \lambda)$ maps the Weyl solution matrix $\tilde \Phi(x, \lambda)$ of the problem $\tilde {\mathcal L}$ to the Weyl solution matrix $\Phi(x, \lambda)$ of the problem $\mathcal L$. 	
	
	Lemma~4.6 of \cite{Bon1} and the proof of Theorem~2 in \cite{Bon3} imply the following proposition.
	
	\begin{proposition} \label{prop:P}
		Suppose that $\mathcal{L}, \tilde{\mathcal{L}} \in \mathcal{W}, \Lambda=\tilde{\Lambda}$ and $\mathcal{N}(\lambda_0)=\tilde{\mathcal{N}}(\lambda_0)$ for all $\lambda_0 \in \Lambda$. Then the matrix of spectral mappings $\mathcal{P}(x, \lambda)$ does not depend on $\lambda$, that is, $\mathcal{P}(x, \lambda) \equiv \mathcal P(x)$. Moreover, $\mathcal{P}(x)$ is a unit lower-triangular matrix functions and the following relation holds:
		\begin{equation}\label{PF}
		\mathcal{P}'(x)+\mathcal{P}(x)\tilde{F}(x)=F(x)\mathcal{P}(x), \quad x\in (0,1).
		\end{equation}
	\end{proposition}
	
	\section{Proofs of main theorems} \label{sec:proofs}
	In this section, we prove Theorems~\ref{thm:main} and~\ref{thm:main222}. The main idea of the proof of Theorem~\ref{thm:main} is that the given spectra $\mathfrak S_{ij}$ uniquely specify the entries of the Weyl-Yurko matrix except for $m_{41}(\lambda)$. The entry $m_{41}(\lambda)$ is uniquely specified up to a constant. Therefore, we cannot uniquely recover $\tau_2$. Anyway, we apply the method of spectral mappings and use the special structure of the associated matrix \eqref{ident3} to achieve the assertion of Theorem~\ref{thm:main}. The proof of Theorem~\ref{thm:main222} also relies on the special structure of the associated matrix as well as of the boundary conditions \eqref{bc-Bar}.
	
	\begin{proof}[Proof of Theorem~\ref{thm:main}]
	Suppose that the two problems $\mathcal{L}$ and $\tilde{\mathcal{L}}$ of class $\mathcal W$ have the same spectra $\mathfrak{S}_{12}=\tilde{\mathfrak{S}}_{12},\mathfrak{S}_{13}=\tilde{\mathfrak{S}}_{13},\mathfrak{S}_{23}=\tilde{\mathfrak{S}}_{23}$. The proof consists of four steps.
	
	\textbf{Step 1.}\label{spe}
 	Let us show that $m_{jk}(\lambda) = \tilde m_{jk}(\lambda)$ for $(j,k) = (2,1), (3,1), (3,2), (4,2)$.
	
	By Lemma \ref{char} and the definition of the Weyl-Yurko matrix, we have $m_{42}(\lambda)=\tilde{m}_{42}(\lambda)$ and $m_{32}(\lambda)=\tilde{m}_{32}(\lambda)$. We also have $m_{43}(\lambda)=m_{21}(\lambda)$ via (\ref{real1}), so we only need to prove $m_{21}(\lambda)=\tilde{m}_{21}(\lambda)$ and $m_{31}(\lambda)=\tilde{m}_{31}(\lambda)$.
	
	We calculate the residues of the both sides in (\ref{real2}) and so get the following relation:
	\begin{equation*}
		m_{42\langle-1\rangle}(\lambda_{n2})=m_{32\langle-1\rangle}(\lambda_{n2})m_{21}(\lambda_{n2}), \quad n\geq 1,
	\end{equation*}
	where $\{\lambda_{n2}\}_{n\geq1}$ are the zeros of the function $\Delta_{22}$.
	Hence, we have
	\begin{equation} \label{equal}
		m_{21}(\lambda_{n2})=\tilde{ m}_{21}(\lambda_{n2}).
	\end{equation}
	
	Construct the function
	\begin{equation*}
		G(\lambda)=\Delta_{11}(\lambda)\tilde{\Delta}_{21}(\lambda)-\Delta_{21}(\lambda)\tilde{\Delta}_{11}(\lambda),
	\end{equation*}
	Note that $\mathcal L \in \mathcal W$ implies $\Delta_{11}(\lambda_{n2}) \ne 0$, $n \ge 1$. Therefore, it follows from (\ref{equal}) that 
	\begin{equation*}
		G(\lambda_{n2})=0, \quad n\geq 1.
	\end{equation*}
	Hence, the function $K(\lambda)=\dfrac{G(\lambda)}{\Delta_{22}(\lambda)}$ is entire in $\lambda$.
	
	By Proposition \ref{delta'}, we get the following asymptotic estimates on the ray $\arg\rho=\frac{\pi}{4}$:
	\begin{equation*}
			{\Delta}_{11}(\lambda)=O(\rho^{-3}e^\rho),\quad
			{\Delta}_{21}(\lambda)=O(\rho^{-2}e^\rho),\quad
			|{\Delta}_{22}(\lambda)|\geq C_{\delta}|\rho|^{-3}e^{\Im \rho+\Re \rho}.
	\end{equation*}
	
	Consequently, we obtain
	\begin{equation*}
		|K(\lambda)|\leq|\rho|^{-2}|e^{\Re\rho-\Im\rho}|=|\rho|^{-2}\rightarrow 0,\;\;\;\arg\rho=\frac{\pi}{4},\;|\rho|\rightarrow\infty.
	\end{equation*}
	Hence
	\begin{equation*}
		|K(\lambda)|\rightarrow0,\;\;\;\arg\lambda=\pi,\;\lambda\rightarrow-\infty.
	\end{equation*}
	The entire functions ${\Delta}_{21}(\lambda)$, ${\Delta}_{22}(\lambda)$ and ${\Delta}_{11}(\lambda)$ have the order $\frac{1}{4}$ and so does the function $K(\lambda)$. Applying Phragmen-Lindel\"of's theorem \cite{BFY},  we can get $K(\lambda) \equiv 0,\;\lambda\in \mathbb{C}$. Hence $G(\lambda) \equiv 0,\;\lambda\in \mathbb{C}$ and so $m_{21}(\lambda)=\tilde{m}_{21}(\lambda)$.
	Taking (\ref{real2}) into account, we also get $m_{31}(\lambda) = \tilde m_{31}(\lambda)$.

	\textbf{Step 2.} Let us prove that, if $m_{21}(\lambda)=\tilde{m}_{21}(\lambda)$,
	$m_{42}(\lambda)=\tilde{m}_{42}(\lambda)$, then $m_{41\langle-1\rangle}(\lambda_{n1})=\tilde{m}_{41\langle-1\rangle}(\lambda_{n1})$, where $\lambda \in \mathbb{C}$ and $\{\lambda_{n1}\}_{n\geq 1}$ are the zeros of $\Delta_{11}(\lambda)$.
	
	Consider the entires of the matrix $\mathcal{N}(\lambda)=[\mathcal{N}_{kj}(\lambda)]_{k,j=1}^4.$
	By Proposition \ref{prop:N} we know that every element in $\mathcal{N}(\lambda_{n1})$ is zero except $\mathcal N_{21}(\lambda_{n1})$ and $\mathcal N_{43}(\lambda_{n1})$.
	Using (\ref{N}), we get the equation
	\begin{equation*}
		\begin{split}
			&\begin{pmatrix} 1 & 0 & 0 & 0\\ m_{21\langle0\rangle}(\lambda_{n1}) & 1 & 0 & 0 \\ m_{31\langle0\rangle}(\lambda_{n1}) & m_{32}(\lambda_{n1}) & 1 & 0 \\ m_{41\langle0\rangle}(\lambda_{n1}) & m_{42}(\lambda_{n1}) & m_{43\langle0\rangle}(\lambda_{n1}) & 1\end{pmatrix}\begin{pmatrix} 0 & 0 & 0 & 0\\ \mathcal N_{21}(\lambda_{n1}) & 0 & 0 & 0 \\ 0 & 0 & 0 & 0 \\ 0 & 0 & \mathcal N_{43}(\lambda_{n1}) & 0\end{pmatrix}\\&=\begin{pmatrix} 0 & 0 & 0 & 0\\ m_{21\langle-1\rangle}(\lambda_{n1}) & 0 & 0 & 0 \\ m_{31\langle-1\rangle}(\lambda_{n1}) & 0 & 0 & 0 \\ m_{41\langle-1\rangle}(\lambda_{n1}) & 0 & m_{43\langle-1\rangle}(\lambda_{n1}) & 0\end{pmatrix}.
		\end{split}
	\end{equation*}
	Consequently
	\begin{equation*}
		\mathcal N_{21}(\lambda_{n1})=m_{21\langle-1\rangle}(\lambda_{n1}),\quad \mathcal N_{21}(\lambda_{n1})m_{42}(\lambda_{n,1})=m_{41\langle-1\rangle}(\lambda_{n1}). 
	\end{equation*}
	Hence
	$$m_{41\langle-1\rangle}(\lambda_{n1})=m_{21\langle-1\rangle}(\lambda_{n1})m_{42}(\lambda_{n1}).$$
	Therefore, we have $m_{41\langle-1\rangle}(\lambda_{n1})=\tilde{m}_{41\langle-1\rangle}(\lambda_{n1})$.
	
	\textbf{Step 3.} Let us prove that $\Lambda=\tilde{\Lambda}$ and $\mathcal{N}(\lambda_0)=\tilde{\mathcal{N}}(\lambda_0)$ for all $\lambda_0 \in \Lambda$.
	
	The equality $\Lambda=\tilde{\Lambda}$ is obvious since  $\Lambda$ and $\tilde{\Lambda}$ are the poles of the functions $\{m_{ij}\}_{(i,j)=(2,1),(3,2),(4,3)}$ and $\{\tilde{m}_{ij}\}_{(i,j)=(2,1),(3,2),(4,3)}$, respectively.
	Consider the entires of the matrix $\mathcal{N}(\lambda_0)=[\mathcal{N}_{kj}(\lambda_0)]_{k,j=1}^4.$
    By Proposition \ref{prop:N}, we know that every element in $\mathcal{N}(\lambda_0)$ is zero except $\mathcal{N}_{21}(\lambda_{0})$, $\mathcal{N}_{32}(\lambda_{0})$ and $\mathcal{N}_{43}(\lambda_{0})$.
	
	Taking (\ref{N}) into account, we obtain that $$\mathcal{N}_{21}(\lambda_{0})=m_{21\langle-1\rangle}(\lambda_{0}), \quad \mathcal{N}_{32}(\lambda_{0})=m_{32\langle-1\rangle}(\lambda_{0}), \quad \mathcal{N}_{43}(\lambda_{0})=m_{43\langle-1\rangle}(\lambda_{0}).$$ Therefore we have $\mathcal{N}(\lambda_0)=\tilde{\mathcal{N}}(\lambda_0)$ for all $\lambda_0 \in \Lambda$.
	
	\textbf{Step 4.} \label{uni} Let us show that $\Lambda=\tilde{\Lambda}$ and $\mathcal{N}(\lambda_0)=\tilde{\mathcal{N}}(\lambda_0)$ for all $\lambda_0 \in \Lambda$ yield $\tilde \tau_1(x) = \tau_1(x)$ and $\tilde \tau_2(x) = \tau_2(x) + cx$ a.e. on $(0,1)$.
	
	Consider the matrix of spectral mappings $\mathcal P(x, \lambda)$ defined by \eqref{defP}.
	It can be shown that
	\begin{equation}\label{P}
		\mathcal{P}(x,\lambda)=\Phi(x,\lambda)J_{0}^{-1}(\tilde \Phi(x,\lambda))^{T}J.
	\end{equation}

	By virtue of Proposition~\ref{prop:P}, $\mathcal P(x, \lambda) \equiv \mathcal P(x)$, where $\mathcal P(x)$ is a unit-lower triangular matrix.
	Therefore, taking (\ref{M'}) and (\ref{P}) into account, we get
	\begin{equation}\label{PF'}
		\mathcal{P}(0)=M(\lambda)J_{0}^{-1}\tilde{M}(\lambda)J.
	\end{equation}
	According to (\ref{PF'}), we have 
	\begin{align*}
		\mathcal{P}(0)=\begin{pmatrix} 1 & 0 & 0 & 0\\ 0  & 1 & 0 & 0 \\ 0 & 0 & 1 & 0 \\ \tilde{m}_{41}-m_{41} & 0 & 0 & 1\end{pmatrix}.
	\end{align*}
	
	According to Proposition~\ref{prop:P}, the relation \eqref{PF} holds.
	Substituting (\ref{ident3}) into 
	(\ref{PF}), we get
	\begin{equation*}
	\begin{split}
	\begin{pmatrix}
		0 & 0 & 0 & 0\\ \mathcal{P}'_{21}  & 0 & 0 & 0 \\ \mathcal{P}'_{31} & \mathcal{P}'_{32} & 0 & 0 \\ \mathcal{P}'_{41} & \mathcal{P}'_{42} &\mathcal{P}'_{43} & 0
	\end{pmatrix}&+\begin{pmatrix}
	1 & 0 & 0 & 0\\ \mathcal{P}_{21}  & 1 & 0 & 0 \\ \mathcal{P}_{31} & \mathcal{P}_{32} & 1 & 0 \\ \mathcal{P}_{41} & \mathcal{P}_{42} &\mathcal{P}_{43} & 1
\end{pmatrix}\begin{pmatrix} 0 & 1 & 0 & 0\\ -\tilde{\tau}_{2} & \tilde{\tau}_{1} & 1 & 0 \\ \tilde{\tau}_{1}\tilde{\tau}_{2} & -\tilde{\tau}_{1}^{2}+2\tilde{\tau}_{2} & -\tilde{\tau}_1 & 1 \\ \tilde{\tau}_{2}^{2} & -\tilde{\tau}_1\tilde{\tau}_2  & -\tilde{\tau_{2}} & 0\end{pmatrix}\\
&=\begin{pmatrix} 0 & 1 & 0 & 0\\ -\tau_{2} & \tau_{1} & 1 & 0 \\ \tau_{1}\tau_{2} & -\tau_{1}^{2}+2\tau_{2} & -\tau_{1} & 1 \\ \tau_{2}^{2} & -\tau_{1}\tau_{2} & -\tau_{2} & 0\end{pmatrix}\begin{pmatrix}
	1 & 0 & 0 & 0\\ \mathcal{P}_{21}  & 1 & 0 & 0 \\ P_{31} & \mathcal{P}_{32} & 1 & 0 \\ \mathcal{P}_{41} & \mathcal{P}_{42} &\mathcal{P}_{43} & 1
\end{pmatrix}.
\end{split}
	\end{equation*}

	From which, we obtain $\mathcal{P}_{32}'=\mathcal{P}_{41}'=0$, $\mathcal{P}_{32}=\tilde{\tau}_{1}-\tau_{1}$, $\mathcal{P}_{42}=\tilde{\tau}_{2}-\tau_{2}$, $\mathcal{P}_{42}'=-\mathcal{P}_{41}$. 
	Then we have $\mathcal{P}_{32}=\mathcal{P}_{32}(0)=0,\;\mathcal{P}_{41}=\mathcal{P}_{41}(0)=\tilde{m}_{41}-m_{41}=c,\;\mathcal{P}_{42}'=-c,$ where $c$ is a constant.
	Hence, we obtain $\tilde{\tau}_{1}=\tau_{1}$, $\tilde{\tau}_{2}=\tau_{2}+cx$ in $L_{2}(0,1)$, which concludes the proof.
	\end{proof}

\begin{proof}[Proof of Theorem~\ref{thm:main222}]
	Recall that the matrices $F(x)$ and $\tilde{F}(x)$ are defined by (\ref{ident3}). Suppose that $\tilde{\tau}_{1}=\tau_{1}$, $\tilde{\tau}_{2}=\tau_{2}+cx$. Then, the quasi-derivatives equal
	\begin{equation*}
		\begin{split}
			y^{[4]}&=y^{(4)}-\tau_{1}''y'-\tau_{1}'y''+\tau_{2}''y\\&=y^{(4)}-\tilde{\tau}_{1}''y'-\tilde{\tau}_{1}'y''+\tilde{\tau}_{2}''y
		\end{split}
	\end{equation*}
	and $l(y)=y^{[4]}$. Consider the boundary conditions corresponding to $\mathfrak{S}_{12}$:
	\begin{equation*}
		\begin{split}
			&U_{1}(y)=y^{[0]}(0)=y(0)=0,\;\;U_{2}(y)=y^{[1]}(0)=y'(0)=0,\\
			&V_{3}(y)=y^{[2]}(1)=y''(1)-\tau_{1}y'(1)=y''(1)-\tilde{\tau}_{1}y'(1)=0,\\
			&V_{4}(y)=y(1)=0,
		\end{split}
	\end{equation*}
	then we have $\mathfrak{S}_{12}=\tilde{\mathfrak{S}}_{12}$. By using the same method, we can obtain $\mathfrak{S}_{13}=\tilde{\mathfrak{S}}_{13}$ and $\mathfrak{S}_{23}=\tilde{\mathfrak{S}}_{23}$.\end{proof}

\section{Lower singularity orders}\label{sec:different}

The case $p \in W_2^{-1}(0,1)$ and $q \in W_2^{-2}(0,1)$ corresponds to the highest singularity orders for which the Mirzoev-Shkalikov regularization \cite{Mir1} of the differential expresssion \eqref{ident2} exists. In this section, we study the case of lower singularity orders of the coefficients $p$ and $q$. In this case, the inverse problem can be considered from the two viewpoints. First, it can be treated as the special case of Problem~\ref{ip:main} related to the associated matrix~\eqref{ident3}. Then, the results for lower singularity orders can be derived as corollaries of Theorem~\ref{thm:main}. Second, we can use other types of associated matrices. From these two viewpoints, we can obtain different spectra. Anyway, the uniqueness of the inverse problem solution holds. We illustrate this by the following two examples.

\begin{example}\label{e1}
	Suppose that $p\in W_2^{-1}(0,1)$, $\tau_{1}'=p$, $\tau_{1}\in L_{2}(0,1)$, $q\in W_{2}^{-1}(0,1), \sigma_{2}'=q$, $\sigma_{2}\in L_{2}(0,1).$ If $\mathfrak{S}_{12}=\tilde{\mathfrak{S}}_{12}$, $\mathfrak{S}_{13}=\tilde{\mathfrak{S}}_{13}$, $\mathfrak{S}_{23}=\tilde{\mathfrak{S}}_{23}$, then $\tau_{1}=\tilde{\tau}_{1},$ $\sigma_{2}=\tilde{\sigma}_{2}+c$.
\end{example}
\begin{enumerate}
	\item From the first viewpoint, put $$\tau_2(x)=\int_{0}^{x}\sigma_{2}(t)dt+\tau_2(0),$$ then $\tau_2''=q$. Obviously, $\tau_2\in L_{2}(0,1)\cap AC[0,1]$. Consider the quasi-derivatives induced by the associated matrix \eqref{ident3}. In this case, we have
	\begin{equation}\label{boundary}
		\begin{split}
			&U_{1}(y)=y^{[0]}(0)=y(0)=0,\;\;U_{2}(y)=y^{[1]}(0)=y'(0)=0,\\
			&U_{3}(y)=y^{[2]}(0)=y''(0)+\tau_{2}(0)y(0)-\tau_{1}(0)y'(0)=0,\\&V_{3}(y)=y^{[2]}(1)=y''(1)+\tau_{2}(1)y(1)-\tau_{1}(1)y'(1)=0,\\
			&V_{4}(y)=y(1)=0.
		\end{split}
	\end{equation}

	By Theorem \ref{thm:main}, we have $\tau_{2}=\tilde{\tau}_{2}+cx$, then we obtain $\tau_{2}'=\tilde{\tau}_{2}'+c$, that is, $\sigma_{2}=\tilde{\sigma}_{2}+c$.
	\item From the second viewpoint, we use the following associated matrix for regularization of the differential expression (\ref{ident2}) with $p,q\in W_2^{-1}(0,1)$:
	\begin{equation*}
		\begin{split}
			F(x)=\begin{pmatrix}
				0 & 1 & 0 & 0\\ 0  & \tau_1 & 1 & 0 \\ -\sigma_2 & -\tau_1^2 & -\tau_1 & 1 \\ 0 & \sigma_2 &0 & 0
			\end{pmatrix}.
		\end{split}
	\end{equation*}
	In this case, we have
		\begin{equation}\label{boundary2}
		\begin{split}
			&U_{1}(y)=y^{[0]}(0)=y(0)=0,\;\;U_{2}(y)=y^{[1]}(0)=y'(0)=0,\\
			&U_{3}(y)=y^{[2]}(0)=y''(0)-\tau_{1}(0)y'(0)=0,\\&V_{3}(y)=y^{[2]}(1)=y''(1)-\tau_1(1)y'(1)=0,\\
			&V_{4}(y)=y(1)=0.
		\end{split}
	\end{equation}
	Following the proof strategy of Theorem~\ref{thm:main}, one can easily show that $\tau_{1}=\tilde{\tau}_{1},\;\sigma_{2}=\tilde{\sigma}_{2}+c$.
\end{enumerate}
 Taking \eqref{boundary} \eqref{boundary2}  into account, we discover that the spectra $\mathfrak{S}_{12}$ and $\mathfrak{S}_{13}$  are unchanged, and  $\mathfrak{S}_{23}$ is unchanged if $\tau_2(0)=0$. Obviously, after improving the regularity, although the boundary conditions may change and the corresponding spectra are no longer the same, we can still uniquely recover the potential from the three spectra.
\begin{example}\label{e2}
Suppose that $p,q\in L_{1}(0,1)$. If $\mathfrak{S}_{12}=\tilde{\mathfrak{S}}_{12}$, $\mathfrak{S}_{13}=\tilde{\mathfrak{S}}_{13}$, $\mathfrak{S}_{23}=\tilde{\mathfrak{S}}_{23}$, then $p=\tilde{p},$ $q=\tilde{q}$.
\end{example}
\begin{enumerate}
\item From the first viewpoint, similarly to Example \ref{e1}, we have the boundary conditions \eqref{boundary},
where $\tau_1'=p,\;\tau_2''=q.$ By Theorem \ref{thm:main}, we have $\tau_{1}=\tilde{\tau}_{1},\;\tau_{2}=\tilde{\tau}_{2}+cx$, so  $p=\tilde{p},\;q=\tilde{q}$.
\item From the second viewpoint, we define the associated matrix as follows:
\begin{equation} \label{matrL1}
	\begin{split}
		F(x)=\begin{pmatrix}
			0 & 1 & 0 & 0\\ 0  & 0 & 1 & 0 \\ 0 & p & 0 & 1 \\ -q & 0 &0 & 0
		\end{pmatrix}.
	\end{split}
\end{equation}
We have the boundary conditions
\begin{equation}\label{boundary4}
	\begin{split}
		&U_{1}(y)=y^{[0]}(0)=y(0)=0,\;\;U_{2}(y)=y^{[1]}(0)=y'(0)=0,\\
		&U_{3}(y)=y^{[2]}(0)=y''(0)=0,\\&V_{3}(y)=y^{[2]}(1)=y''(1)=0,\;
		V_{4}(y)=y(1)=0.
	\end{split}
\end{equation} 
repeating the arguments of the uniqueness theorem proof for the associated matrix \eqref{matrL1}, we show that $p(x) = \tilde p(x)$ and $q(x) = \tilde q(x)$ a.e. on $(0,1)$.
\end{enumerate}

Considering \eqref{boundary} and \eqref{boundary4}, we see that  $\mathfrak{S}_{12}$ is unchanged if $\tau_1(1)=0$, $\mathfrak{S}_{13}$ is unchanged if $\tau_1(0)=\tau_1(1)=0$, and $\mathfrak{S}_{23}$ is unchanged if $\tau_2(0)=\tau_1(1)=0$. Note that the boundary conditions \eqref{boundary4} exactly coincide with the boundary conditions \eqref{four}. Therefore, our uniqueness theorem generalizes the original result of Barcilon.


\begin{thebibliography}{99}
		\bibitem{bada} Badanin, A., Korotyaev, E. L., Resonances for Euler–Bernoulli operator on the half-line. Journal of Differential Equations, 2017, 263(1): 534-566.
				
		\bibitem{Bar1} Barcilon, V., On the uniqueness of inverse eigenvalue problems. Geophysical Journal International, 1974, 38(2): 287-298.
		
		\bibitem{Bar2} Barcilon, V., On the solution of inverse eigenvalue problems of high orders. Geophysical Journal International, 1974, 39(1): 143-154.
		
		\bibitem{Bon5} Bondarenko, N. P., Solving an inverse problem for the Sturm-Liouville operator with a singular potential by Yurko's method. Tamkang J. Math., 2021, 52(1), 125-154.
		
		\bibitem{Bon4} Bondarenko, N. P., Inverse problem solution and spectral data characterization for the matrix Sturm–Liouville operator with singular potential. Analysis and Mathematical Physics, 2021, 11: 1-26.

		\bibitem{Bon2} Bondarenko, N. P., Inverse spectral problems for arbitrary-order differential operators with distribution coefficients. Mathematics, 2021, 9(22): 2989. 
				
		\bibitem{Bon1} Bondarenko, N. P., Reconstruction of higher-order differential operators by their spectral data. Mathematics, 2022, 10(20): 3882. 

		\bibitem{Bon3} Bondarenko, N. P., Linear differential operators with distribution coefficients of various singularity orders. Mathematical Methods in the Applied Sciences, 2023, 46(6): 6639–6659.
		
		\bibitem{Bon7} Bondarenko, N. P., Inverse spectral problem for the third-order differential equation. arXiv: 2303.13124.
		
		\bibitem{Bon6} Bondarenko, N. P., Spectral data asymptotics for the higher-order differential operators with distribution coefficients. Journal of Mathematical Sciences, 2023, 1-22.
		
		\bibitem{Bor} Borg, G., Eine umkehrung der Sturm-Liouvilleschen eigenwertaufgabe. Acta Mathematica, 1946, 78(1): 1-96.
		
		\bibitem{BFY} Buterin, S., A., Freiling, G., Yurko V.\;A., Lectures in the theory of entire functions. Schriftenriehe der Fakultat fur Matematik, Duisbug-Essen University. SM-UDE-779, 2014.
		
		\bibitem{Cau} Caudill Jr, L. F., Perry, P. A., Schueller, A. W., Isospectral sets for fourth-order ordinary differential operators. SIAM journal on mathematical analysis, 1998, 29(4): 935-966.
		
		\bibitem{FY01} Freiling, G., Yurko, V., Inverse Sturm-Liouville Problems and Their Applications. Nova Science Publishers, Huntington, NY, 2001.
		
		\bibitem{Fre2} Freiling, G., Ignatiev, M. Y., Yurko, V. A., An inverse spectral problem for Sturm-Liouville operators with singular potentials on star-type graph. Proc. Symp. Pure Math., 2008, 77: 397-408.
		
		\bibitem{GL51}
		Gel'fand, I. M., Levitan, B. M., On the determination of a differential equation from its spectral function. Izv. Akad. Nauk SSSR, Ser. Mat., 1951, 15(4): 309-360.
		
		\bibitem{Glad} Gladwell, G. M. L., Inverse Problems in Vibration, Second Edition. Solid Mechanics and Its Applications, Springer, Dordrecht, 2005, 119.
		
		\bibitem{Hry3} Hryniv, R. O., Mykytyuk, Y. V., Half-inverse spectral problems for Sturm–Liouville operators with singular potentials. Inverse Problems, 2004, 20(5): 1423.
		
		
		\bibitem{Hry1} Hryniv, R.O., Mykytyuk, Y. V., Inverse spectral problems for Sturm–Liouville operators with singular potentials. Inverse Problems, 2003, 19(3): 665-684.
		
		\bibitem{Hry2} Hryniv, R. O., Mykytyuk, Y. V., Inverse spectral problems for Sturm-Liouville operators with singular potentials, II. Reconstruction by two spectra, North-Holland Mathematics Studies. North-Holland, 2004, 197: 97-114.
		
	
		\bibitem{xx} Jiang, X., Xu, X., An inverse spectral problem for a fourth-order Sturm–Liouville operator based on trace formulae. Applied Mathematics Letters, 2021, 111: 106654.
		
		\bibitem{Krav20} Kravchenko, V. V., Direct and Inverse Sturm-Liouville Problems. Birkh\"auser, Cham, 2020.
		
		
		\bibitem{Lev87}
		Levitan, B. M., Inverse Sturm-Liouville Problems. VNU Sci. Press, Utrecht, 1987.

		\bibitem{Mar86} Marchenko, V. A., Sturm-Liouville Operators and Their Applications. Birkh\"auser, Basel, 1986.
				
		\bibitem{Mc2} McLaughlin, J. R., Analytical methods for recovering coefficients in differential equations from spectral data. SIAM Review, 1986, 28(1): 53-72.
		
		\bibitem{Mc1} McLaughlin, J. R., On constructing solutions to an inverse Euler–Bernoulli problem. Inverse Problems of Acoustic $\&$ Elastic Waves, 1984: 341-347.
		
		\bibitem{Mik} Mikhlin, S. G., Chambers, L.\;I.\;G., Variational Methods in Mathematical Physics. Oxford: Pergamon Press, 1964.
		
		\bibitem{Mir1} Mirzoev, K.A., Shkalikov, A.A., Differential operators of even order with distribution coefficients. Mathematical Notes, 2016, 99: 779-784.
		
		\bibitem{Mir2} Mirzoev, K.A., Shkalikov A.A., Ordinary differential operators of odd order with distribution coefficients. arXiv:1912.03660, 2019.
		
		\bibitem{mora} Morassi, A., Exact construction of beams with a finite number of given natural frequencies. Journal of Vibration and Control, 2015, 21(3): 591-600.
		
		\bibitem{PK} Papanicolaou, V. G., Kravvaritis, D., An inverse spectral problem for the Euler-Bernoulli equation for the vibrating beam. Inverse Problems, 1997, 13(4): 1083.
		
		\bibitem{pere} Perera, U., B\"ockmann, C., Solutions of direct and inverse even-order Sturm-Liouville problems using Magnus expansion. Mathematics, 2019, 7(6): 544.
		
		\bibitem{poly} Polyakov, D. M., Spectral analysis of a fourth-order nonselfadjoint operator with nonsmooth coefficients. Siberian Mathematical Journal, 2015, 56(1).
		
		\bibitem{Pol} Polyakov, D. M., Spectral estimates for the fourth-order operator with matrix coefficients. Computational Mathematics and Mathematical Physics, 2020, 60: 1163-1184.
		
		\bibitem{Sav} Savchuk. A. M., Shkalikov. A.A., Asymptotic analysis of solutions of ordinary differential equations with distribution coefficients. Sbornik: Mathematics, 2020, 211(11): 1623.
		
		\bibitem{ugur} U\v{g}urlu, E., Bairamov, E., Fourth order differential operators with distributional potentials. Turkish Journal of Mathematics, 2020, 44(3): 825-856.
		
		\bibitem{Vlad04} Vladimirov, A. A., On the convergence of sequences of ordinary differential equations. Mathematical Notes, 2004, 75: 877-880.
		
		\bibitem{Yur} Yurko, V. A., Method of spectral mappings in the inverse problem theory. VSP, 2002.

		\bibitem{Yur1} Yurko, V. A., On higher-order differential operators with a singular point. Inverse problems, 1993, 9(4): 495.
	
		\bibitem{Yur2} Yurko, V. A., On higher-order differential operators with a regular singularity. Sbornik Mathematics, 1995, 186(6): 901-928.		
	\end{thebibliography}
\end{document}